\pdfoutput=1

\documentclass[reqno]{amsart}

\usepackage[latin1]{inputenc}

\usepackage{amssymb}
\usepackage{amsmath}
\usepackage{amscd}
\usepackage{amsthm}
\usepackage{amsfonts}
\usepackage{enumerate}
\usepackage{graphicx}
\usepackage{url}
\usepackage[breaklinks=true,hyperref]{hyperref}
\usepackage{amssymb}
\usepackage[]{color}
\usepackage{epsfig}
\usepackage{mathrsfs}

\newtheorem{theorem}{Theorem}[section]
\newtheorem{lemma}[theorem]{Lemma}

\theoremstyle{definition}

\newtheorem{definition}[theorem]{Definition}

\theoremstyle{remark}

\numberwithin{equation}{section}

\newcommand{\kautz}{\text{Kautz}}
\newcommand{\db}{DB}
\newcommand{\nn}{\mathbb{N}}
\newcommand{\zz}{\mathbb{Z}}

\newcommand{\Li}{\mathcal{L}}

\newcommand{\indeg}{\text{indeg}}
\newcommand{\outdeg}{\text{outdeg}}

\author[H. Bidkhori, S. Kishore]{Hoda Bidkhori, Shaunak Kishore}
\date{\today}
\title{Counting the spanning trees of a directed line graph}

\begin{document}

\maketitle
\begin{abstract} The line graph $\Li G$ of a directed graph $G$ has a vertex for every edge of $G$ and an edge for every path of length 2 in $G$. In 1967, Knuth used the Matrix-Tree Theorem to prove a formula for the number of spanning trees of $\Li G$, and he asked for a bijective proof \cite{Kn}. In this paper, we give a bijective proof of a generating function identity due to Levine \cite{Le} which generalizes Knuth's formula. As a result of this proof, we find a bijection between binary de Bruijn sequences of degree $n$ and binary sequences of length $2^{n-1}$. Finally, we determine the critical groups of all the Kautz graphs and de Bruijn graphs, generalizing a result of Levine \cite{Le}. 
\end{abstract} 

\section{Introduction} \label{introduction} 

In a directed graph $G=(V,E)$, each edge $e\in E$ is directed from its source $s(e)$ to its target $t(e)$. The directed line graph $\Li G$ of $G$ with vertex set $E$, and with an edge $(e,f)$ for every pair of edges in $G$ such that $t(e)=s(f)$. A spanning tree of $G$ rooted at a vertex $r$ is an edge-induced subgraph of $G$ in which there is a unique path from $v$ to $r$, for all $v\in V$.

We denote the indegree and outdegree of a vertex $v$ by $\indeg(v)$ and $\outdeg(v)$, respectively, and we denote the number of spanning trees of $G$ by $\kappa(G)$. Knuth proved that if every vertex of $G$ has indegree greater than 0, then 
\[\kappa(\Li G)=\kappa(G)\prod_{v\in V} \outdeg(v)^{\indeg(v)-1}\]
Knuth's proof relied on the Matrix-Tree Theorem. In his paper, he noted that the simple form of this result suggested that a bijective proof was possible, but that it was not at all obvious how to find such a bijection \cite{Kn}.

In fact, there are even stronger relations between $\kappa(\Li G)$ and $\kappa(G)$. Let $\{x_v|v\in V\}$ and $\{x_e|e\in E\}$ be variables indexed by the vertices and edges of $G$.  The vertex and edge generating functions of $G$ are defined as follows, where the sums are taken over all rooted spanning trees $T$ of $G$.
\[\kappa^{edge}(G)=\sum_T \prod_{e\in T}x_e,\ \ \kappa^{vertex}(G)=\sum_T \prod_{e\in T}x_{t(e)}\]
Levine used linear algebraic methods to prove the following generalization of Knuth's result. Our first result in this paper is a bijective proof of Levine's theorem, which yields a bijective proof on Knuth's theorem as a special case.
\\ \\
\noindent \textbf{Theorem~\ref{thm1.1}.} \textit{Let $G=(V,E)$ be a directed graph in which every vertex has indegree greater than 0. Then}
\[\kappa^{vertex}(\mathcal{L}G)=\kappa^{edge}(G)\prod_{v\in V} \left(\sum_{s(e)=v} x_e\right)^{\indeg(v)-1} \]
\\
Using this bijection, we are able to answer the following open question posed by Stanley.
\\ \\ 
\textbf{Exercise 5.73 from \cite{St}.} Let $\mathcal{B}(n)$ be the set of binary de Bruijn sequences of degree $n$, and let $\mathcal{S}_n$ be the set of all binary sequences of length $2^n$. Find an explicit bijection $\mathcal{B}(n)\times \mathcal{B}(n)\rightarrow \mathcal{S}(n)$.
\\ \\ 
The critical group $K(G)$ of a graph $G$ is a finite abelian group whose order is the number of spanning trees of $G$. Critical groups have applications in statistical physics \cite{Ho}, algebraic combinatorics \cite{Le}, and arithmetic geometry \cite{Be}. We review the definition of this group in section \ref{definitions}.

The Kautz graphs $\kautz_n(m)$ and the de Bruijn graphs $\db_n(m)$ are families of iterated line graphs. $\kautz_1(m)$ is the complete directed graph on $m+1$ vertices, without self-loops, and $\db_1(m)$ is the complete graph on $m$ vertices, with self-loops. These families are defined for $n>1$ as follows.
\[\kautz_n(m)=\Li^{n-1}\kautz_1(m),\ \ \db_n(m)=\Li^{n-1}\db_1(m)\]
Levine recently determined $K(\db_n(2))$ and $K(\kautz_n(m))$, where $m$ is prime \cite{Le}. We generalize these results, proving the following characterizations of the critical groups of all the Kautz and de Bruijn graphs.
\\ \\
\noindent \textbf{Theorem~\ref{db}.} \textit{The critical group of $\db_n(m)$ is}
\[K\left(\db_n(m)\right)=\left(\zz_{m^n}\right)^{m-2}\oplus\bigoplus_{i=1}^{n-1} \left(\zz_{m^i}\right)^{m^{n-1-i}(m-1)^2}\]
\\
\noindent \textbf{Theorem~\ref{kautz}.} \textit{The critical group of $\kautz_n(m)$ is}
\[K\left(\kautz_n(m)\right)=\left(\zz_{m+1}\right)^{m-1}\oplus \left(\zz_{m^{n-1}}\right)^{m^2-2}\oplus\bigoplus_{i=1}^{n-2} \left(\zz_{m^i}\right)^{m^{n-2-i}(m-1)^2(m+1)}\]
\\
The rest of this paper is organized as follows. In Section~\ref{definitions} we provide background and definitions. In Section~\ref{bijection}, we introduce a bijection which proves Theorem~\ref{thm1.1}. We apply this bijection in Section~\ref{dbsequence} to construct a bijection betweeen binary de Bruijn sequences of order $n$ and binary sequences of length $2^{n-1}$. Finally, in Section~\ref{kautzsection}, we prove Theorem~\ref{db} and ~\ref{kautz}, giving a complete description of the critical groups of the Kautz and de Bruijn graphs.

\section{Background and definitions} \label{definitions}

In a directed graph $G=(V,E)$, each edge $e\in E$ is directed from its \textit{source} $s(e)$ to its \textit{target} $t(e)$. 

\begin{definition}[Directed line graph] Let $G=(V,E)$ be a directed graph. The \textit{directed line graph} $\mathcal{L}G$ is a directed graph with vertex set $E$, and with an edge $(e,f)$ for every pair of edges $e$ and $f$ of $G$ with $t(e)=s(f)$. \end{definition}

\begin{center}\includegraphics[height=2.4cm]{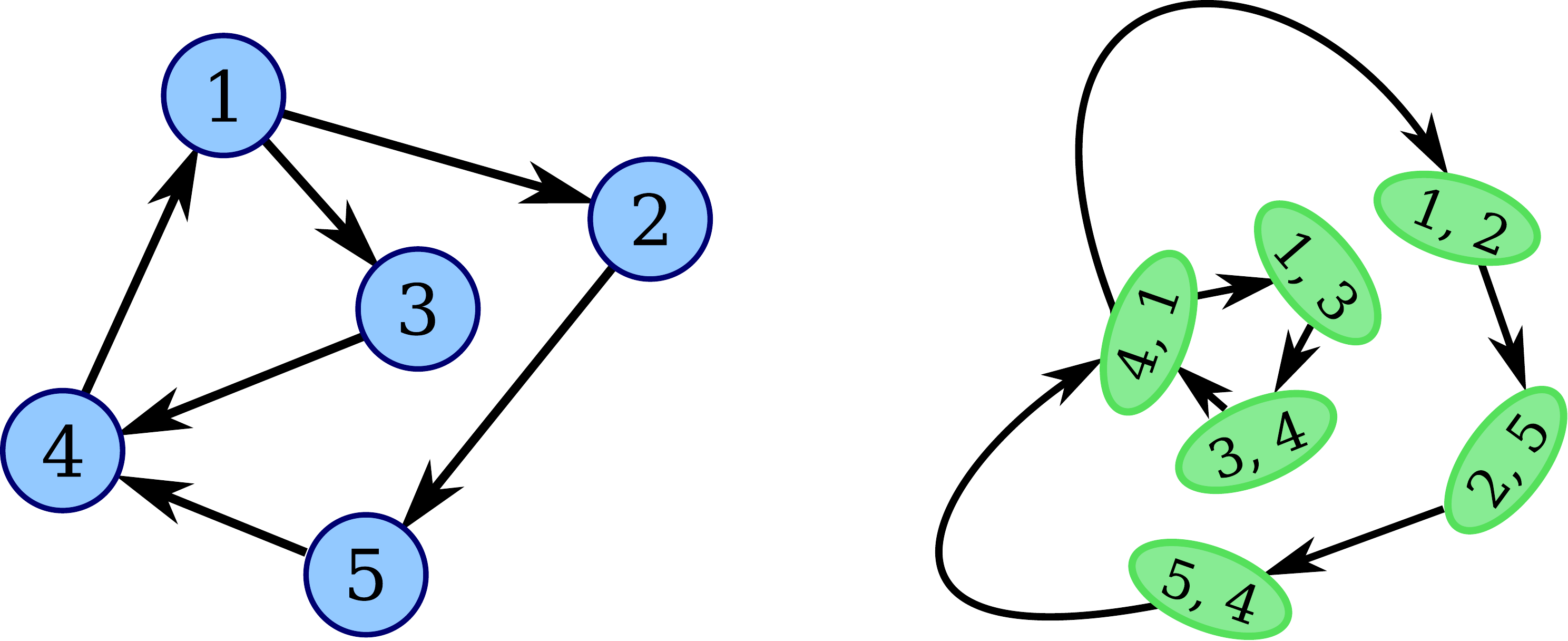}
\\ \textit{Figure 2.1. A directed graph and its line graph.}
\end{center}

At times we may speak of a subset $F$ of $E$ as a subgraphs of $G$ - in this case we mean the subgraph $(V,F)$. If $H$ is a subgraph of $G$ and $v$ is in $H$, we denote the indegree of $v$ in $H$ by $\indeg_H(v)$, and the outdegree by $\outdeg_H(v)$.

\begin{definition}[Oriented spanning tree] Let $G=(V,E)$ be a directed graph. An \textit{oriented spanning tree} of $G$ is an acyclic subgraph of $G$ with a distinguished node, the \textit{root}, in which there is a unique path from every vertex $v\in V$ to the root. We refer to these trees as \textit{spanning trees}. \end{definition}

Let $T$ be a spanning tree of $G$. Every vertex of $G$ has outdegree 1 in $T$, except the root, which has outdegree 0. We denote the number of spanning trees of $G$ by $\kappa(G)$, and the number of spanning trees rooted at $r$ by $\kappa(G,r)$. 

Let $G=(V,E)$ be a strongly-connected directed graph, and let $\zz^V$ be the free abelian group generated by vertices of $G$ -- the group of of formal linear combinations of vertices of $G$. We define $\Delta_v\in \zz^V$, for all $v\in V$, as follows.
\[\Delta_v = \sum_{e\in E\ s.t\ s(e)=v} (t(e)-v) \]
The \textit{sandpile group} $K(G,r)$ with \textit{sink} $r$ is the quotient group 
\[K(G,r)=\zz^V/(r,\Delta_v|v\in V\backslash r)\] 
It is well-known that the order of $K(G,r)$ is $\kappa(G,r)$.

A directed graph $G$ is \textit{Eulerian} if $\indeg(v)=\outdeg(v)$ for all vertices $v$ in $V$. According to Lemma 4.12 of \cite{Ho}, if $G$ is Eulerian, the sandpile groups $K(G,r_1)$ and $K(G,r_2)$ are isomorphic for any two $r_1,r_2$ in $V$. In this case, we call the group the \textit{critical group} $K(G)$.

\begin{definition}[The Laplacian] Let $G=(V,E)$ be a finite directed graph with vertices $v_1,v_2,\ldots v_{|V|}$. The \textit{adjacency matrix} $A(G)$ of $G$ is the $|V|\times |V|$ matrix in which $A(G)_{ij}$ is the multiplicity of the edge $(v_i,v_j)$ in $E$. The \textit{degree matrix} $D(G)$ is the $|V|\times |V|$ diagonal matrix in which $D_{ii}=\outdeg(v_i)$. The \textit{Laplacian} $L(G)$ of $G$ is defined as $A(G)-D(G)$.
\end{definition}
Note that the row vectors of $L(G)$ are the elements $\Delta_v$. We consider $L(G)^T$ as a $\zz$-linear operator on $\zz^V$ -- its image is the subgroup generated by the $\Delta_v$. For a strongly-connected Eulerian graph $G$, the Laplacian has exactly one eigenvalue 0, so for such a graph $G$, we have
\[\zz^V/\text{im}\ L(G)\cong \zz^V/\text{im}\ L(G)^T\cong  \zz\oplus K(G)\]
The following elementary row and column operations on matrices with entries in a ring $R$ are invertible over $R$. 
\begin{enumerate}[-]
\item Permuting two rows (columns)
\item Adding a multiple of a row (column) by an element of $R$ to another row (column)
\item Multiplying the entries of a row (column) by a unit
\end{enumerate}
If $L'$ is obtained from $L(G)$ by invertible row and column operations over $\zz$, then $\zz^V/\text{Im}\ L' \cong \zz^V/\text{Im}\ L(G)$. 

Suppose that $R$ is a principal ideal domain. Under these operations, any matrix with entries in $R$ is equivalent to a matrix in \textit{Smith normal form}. A matrix in this form is diagonal, and its diagonal entries $x_{11},x_{22},\ldots x_{nn}$ are elements of $R$ such that $x_{(i+1)(i+1)}$ is a multiple of $x_{ii}$ for all $i<n$. These entries are called the \textit{invariant factors} of the original integer matrix, and they are unique up to multiplication by units. If the invariant factors of $L(G)$ over $\zz$ are $x_{11},x_{22},\ldots x_{nn}$ then
\[\zz^V/\text{Im}\ L(G) = \bigoplus_{i=1}^n \zz_{x_{ii}}\]
Thus, row-reducing the Laplacian yields information about the critical group.

\section{Counting spanning trees} \label{bijection}
Let $G=(V,E)$ be a directed graph, and let $\{x_v\}_{v\in V}$ and $\{x_e\}_{e\in E}$ be variables indexed by the vertices and edges of $G$. The \textit{edge and vertex generating functions}, which enumerate the spanning trees of $G$, are defined as follows
\[\kappa^{edge}(G)=\sum_T \prod_{e\in T}x_e\]
\[\kappa^{vertex}(G)=\sum_T \prod_{e\in T}x_{t(e)}\]
\noindent where $T$ ranges over all spanning trees of $G$. In this section, we give a bijective proof of the following identity, solving a problem posed by Levine in \cite{Le}

\begin{theorem} \label{thm1.1} Let G=(V,E) be a directed graph in which every vertex has indegree greater than 0. Then:
\begin{equation} \label{kthm} \kappa^{vertex}(\mathcal{L}G)=\kappa^{edge}(G)\prod_{v\in V} \left(\sum_{s(e)=v} x_e\right)^{\indeg(v)-1} \end{equation} \end{theorem}

In order to find a bijection, we adopt the following strategy. We put an arbitrary total order on the edges in $E$.
\begin{enumerate}[-] 
\item We provide a bijection between monomial terms on the right-hand side of Eq. (\ref{kthm}) and \textit{tree arrays}, which are arrays of lists, one list for each vertex $v\in V$. 
\item Then we present a map $\sigma$ that take a tree array to a spanning tree of $\Li G$ which contributes the same term to the left-hand side of Eq. (\ref{kthm}). 
\item Finally, we show that $\sigma$ is bijective by constructing an inverse map $\pi$ which takes a spanning tree of $\Li G$ to a tree array.
\end{enumerate}

We define a \textit{list} to be an ordered tuple of edges. We \textit{append} an element $x$ to a list $l$ by adding $x$ to the end of $l$. We \textit{pop} list $l$ by removing the first element of $l$. We denote the number of times an element $e$ appears in a list $l$ by $N(l,e)$ .

Let $v$ be a vertex of $G$ and let $l'_v$ be a list with $\indeg(v)-1$ elements, all of which are edges with source $v$. We map $l'_v$ to a monomial term of $(\sum_{s(e)=v} x_e)^{\indeg(v)-1}$, as follows.
\[l'_v=(e_1,e_2,\ldots e_{\indeg(v)-1})\rightarrow x_{e_1}x_{e_2}\ldots x_{e_{\indeg(v)-1}}\]
This map provides a bijection between lists $l'_v$ and terms of $(\sum_{s(e)=v} x_e)^{\indeg(v)-1}$. Therefore, a term on the right-hand side of Eq. (\ref{kthm}) corresponds to a choice of spanning tree $T$ of $G$ and a choice of one such list $l'_v$ for each vertex $v$. 

Suppose a monomial term on the right-hand side of Eq. (\ref{kthm}) corresponds to a spanning tree $T$ rooted at $r$ and an array of lists $\langle l'_v\rangle$. For each vertex $v\in V\backslash r$, we obtain $l_v$ by appending the unique edge $e$ in $T$ with source $v$ to the list $l'_v$. We obtain $l_r$ by appending a new variable $\Omega$ to $l'_r$. 

Each list $l_v$ has length $\indeg(v)$, for $v\in V$. We call an array of lists $\langle l_v \rangle_{v\in V}$ obtained in this way a \textit{tree array}. By construction, terms on the right-hand side of Eq. (\ref{kthm}) are in bijection with tree arrays. 

We now define the bijective map $\sigma$, which takes a tree array of $G$ to a spanning tree of $\Li G$. \vspace{.1in}

\noindent \textbf{The bijection $\sigma$:} We start with a tree array $\langle l_v \rangle$ and an empty subgraph $T'$ of $\Li G$. Then we run the following algorithm.
\begin{enumerate}[Step 1.]
\item Let $R$ be the subset of edges $e$ of $G$ for which $N(l_{s(e)},e)=0$ and $\outdeg_{T'}(e)=0$. Let $f$ be the smallest edge in $R$ under the order on $E$. 
\item Pop the first element $g$ from the list $l_{t(f)}$. If $g$ is $\Omega$, then $\sigma(\langle l_v\rangle)=T'$. 
\item Otherwise, $g\in E$ and $s(g)=t(f)$. Add the edge $(f,g)$ to $T'$, and then return to step 1.
\end{enumerate}

We also define a map $\pi$ which takes a spanning tree of $\Li G$ to a tree array of $G$. \vspace{.1in}

\noindent \textbf{The inverse map $\pi$:} We start with a spanning tree $T'$ of $\Li G$, and an empty list $l_v$ at each vertex $v\in V$. This map is given by another algorithm. 
\begin{enumerate}[Step 1.]
\item Let $S$ be the set of leaves of $T'$. Let $f$ be the smallest edge in $S$ under the order on $E$.
\item If $f$ is not the root of $T'$, remove $f$ and its outedge $(f,g)$ from $T'$, and append $g$ to $l_{t(f)}$. Go back to step 2.
\item If $f$ is the root of $T'$, append $\Omega$ to $l_{t(f)}$, and return the array of lists.
\end{enumerate}

As an example, we apply $\sigma$ to a tree array in a small directed graph $G$. We order the edges of $G$ by the lexigraphic order.

\begin{center}\includegraphics[height=2.4cm]{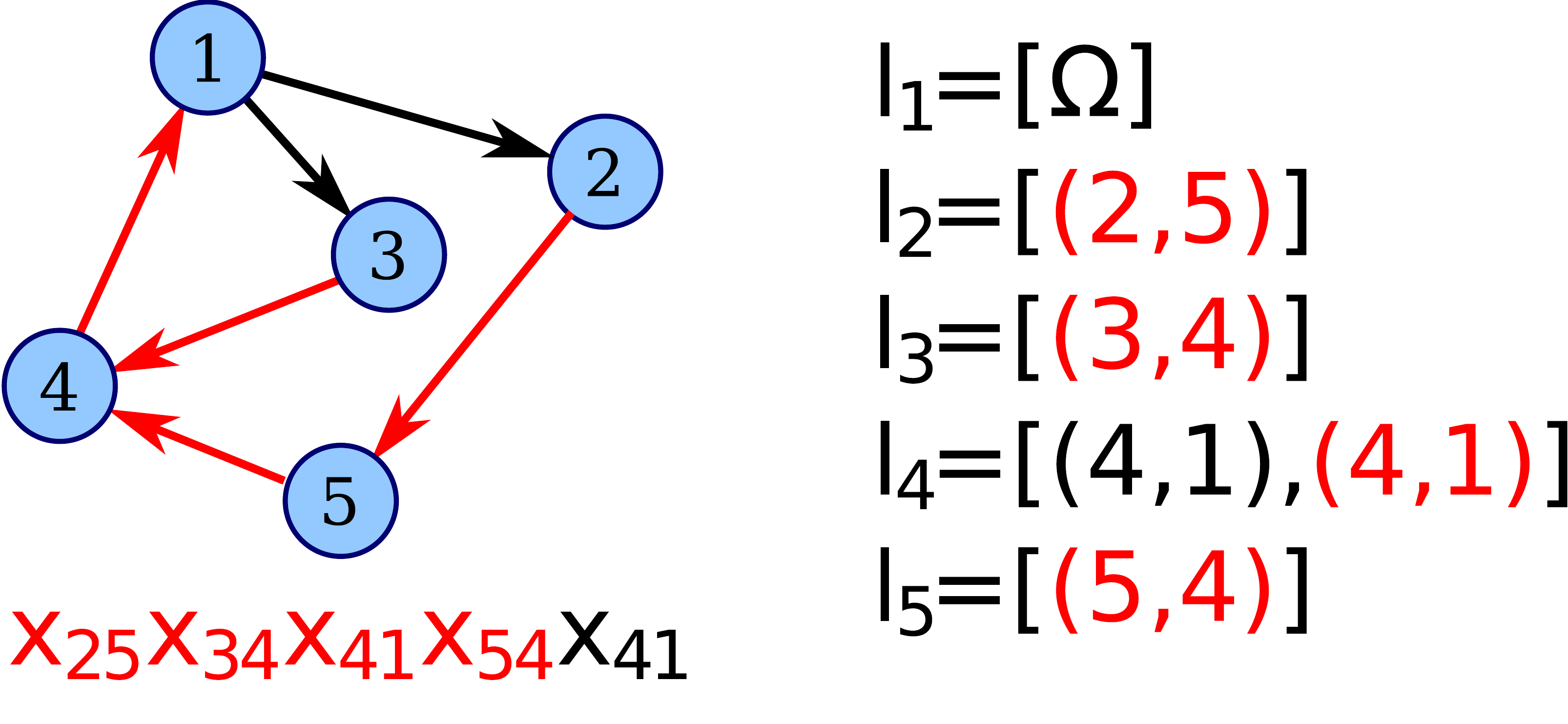}
\\ \textit{Figure 3.1. The graph $G$, with a spanning tree $T$ highlighted in red. Below the graph is a monomial term of $\kappa^{vertex}(G)$, where $x_{ij}$ is the variable for edge $(i,j)$. The tree array corresponding to this term is shown to the right. In the term and the tree array, red elements correspond to edges of the tree.}
\end{center} \vspace{.025in}
\begin{center}\includegraphics[height=2.1cm]{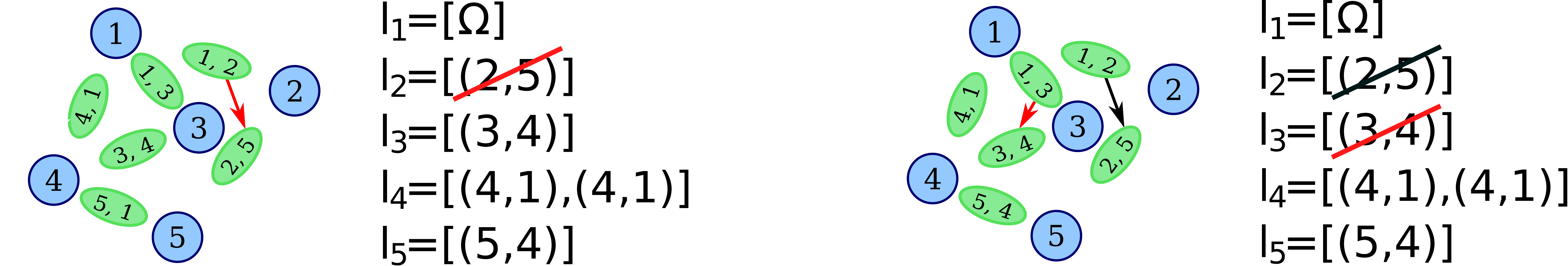}\vspace{.025in}
\\ \textit{Figure 3.2. The first two edges added to $T'$ by the algorithm for $\sigma$. Initially, the edges $(1,2)$ and $(1,3)$ do not appear in the lists. We pop $(2,5)$ from $l_2$ and add the edge $((1,2),(2,5))$ to $T'$. Then the edges $(1,3)$ and $(5,4)$ have outdegree 0 in $T'$ and do not appear in the lists. We pop $(3,4)$ from $l_3$ and add $((1,3),(3,4))$ to $T'$. }
\end{center}
\begin{center}\includegraphics[height=5.1cm]{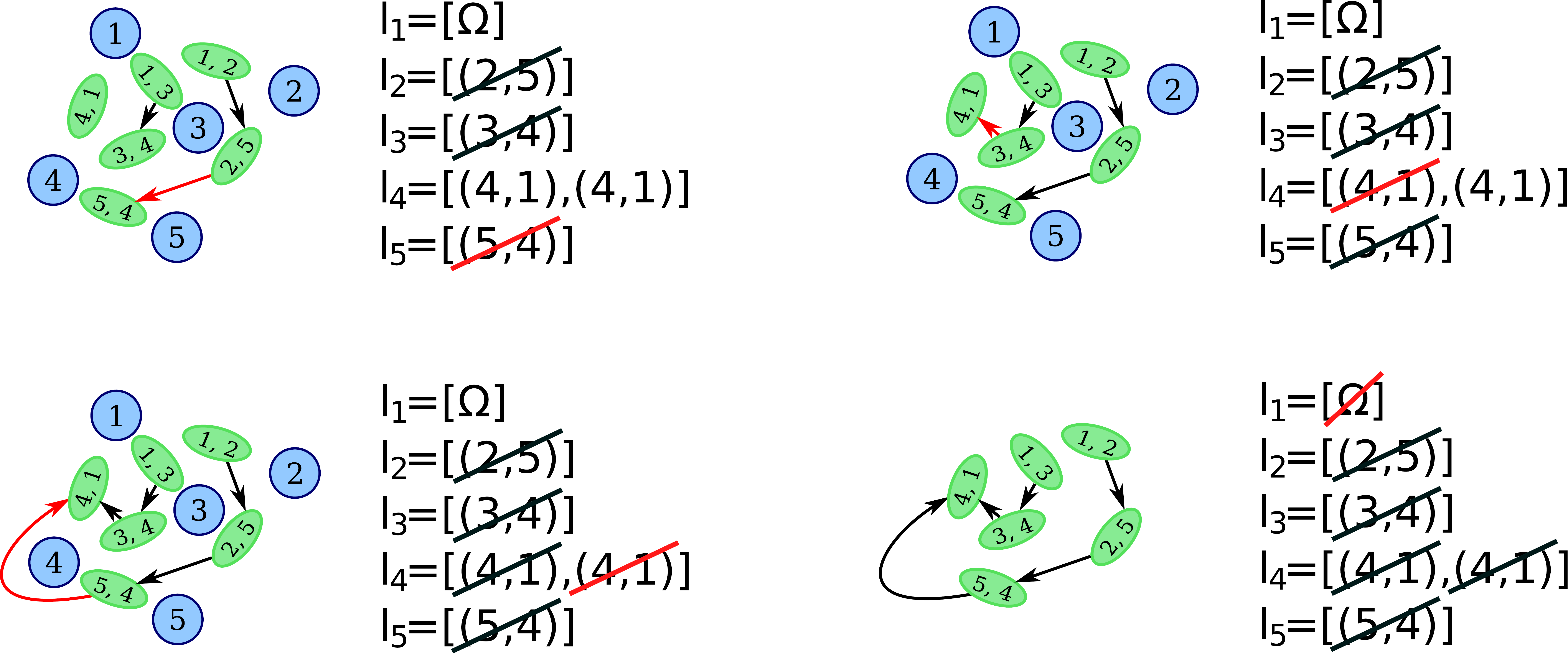}\vspace{.025in}
\\ \textit{Figure 3.3. The last three edges added to $T'$, and the final tree. The last element left in the lists of the tree array is $\Omega$.}
\end{center}

In order to prove Theorem~\ref{thm1.1}, we first prove three lemmas. In the definition of the algorithm for the map $\sigma$, we assumed that the set $R$ is always non-empty in step 1 and that the the list $l_{t(f)}$ is always non-empty in step 2. In Lemma~\ref{welldef}, we show that both assumptions are valid.

\begin{lemma} \label{welldef} The algorithm used to define map $\sigma$ is well-defined: at step 1, the set $R$ is non-empty, and at step 2, the list $l_{t(f)}$ is non-empty. \end{lemma}
\begin{proof} After $k$ edges have been added to $T'$, there are $|E|-k$ elements left in all the lists $l_v$, where one of the elements is $\Omega$. There are $|E|-k-1$ edges left in the lists, but there are $|E|-k$ edges of $G$ which do not have an outedge in $T'$, so $R$ must be non-empty in step 1.

Every time we pop $l_v$, we add an edge $(f,g)$ to $T'$, where $t(f)=v$. When we are at step 2, $\outdeg_{T'}(f)=0$, so at most $\indeg(t(f))-1$ of the elements of $l_{t(f)}$ have been popped. Therefore, the list $l_{t(f)}$ is always nonempty at step 2. The algorithm is well-defined. \end{proof} \vspace{.05in} 

The following lemma shows that $\sigma$ takes a tree array corresponding to a term on the right-hand side of Eq. (\ref{kthm}) to a spanning tree which contributes the same term to the left-hand side.

\begin{lemma} \label{correctterm} Suppose that $\langle l_v \rangle$ is a tree array and that $\sigma(\langle l_v \rangle)=T'$. Then $T'$ is a spanning tree of $\Li G$, and $\indeg_{T'}(e)=N(l_{s(e)},e)$, for all $e\in E$. \end{lemma}
\begin{proof} Let $I(e)$ be the initial value of $N(l_{s(e)},e)$. By the definition of a tree array, the edges which are the last elements of the lists $l_v$ form a spanning tree $T$ of $G$.

We claim that $T'$ is acylic, because the last edge of a cycle is never included in $T'$. While the algorithm is running, suppose that $(e_n,e_1)$ is not an edge of $T'$, and that it completes a cycle $(e_1,e_2),(e_2,e_3),\ldots (e_{n-1},e_n)$ of edges in $T'$. Since $(e_1,e_2)$ was already added to $T'$, $N(l_{s(e_1)},e_1)$ must be 0. Therefore, $(e_n,e_1)$ will never be added to $T'$. 

We say a vertex $v\in V$ is \textit{cleared} if all the elements of its list are popped. Suppose that $e=(v,w)$ is an edge in $T$. The list $l_w$ is cleared when all the edges of $G$ with target $w$ have an outedge in $T'$. Then $w$ can only be cleared after an outedge $(e,f)$ of $e$ is added to $T$. The edge $(e,f)$ can only be added to $T'$ when $N(l_v,e)=0$. Because $e$ is an edge of $T$, it is the last element of $l_v$, so $v$ must be cleared before $w$ can be cleared.

The algorithm terminates when $\Omega$ is popped from $l_r$, which occurs when $r$ is cleared. There is a path $(v,v_1,v_2,\ldots v_k, r)$ in $T$ from any vertex $v$ to $r$. Therefore $r$ can only be cleared after all the vertices on this path are cleared. Thus, all the vertices of $G$ are cleared when the algorithm finishes, so there are $|E|-1$ edges in the subgraph $T'$. 

All the vertices of $\Li G$ has an outedge in $T'$, except one. Since $T'$ is acyclic, it is a spanning tree of $\Li G$. Because $\indeg_{T'}(e)+N(l_{s(e)},e)$ is constant, when the algorithm returns $T'$, $\indeg_{T'}(e)=I(e)$ for all $e\in E$. \end{proof} \vspace{.05in}

In our final lemma, we show that $\pi$ will take a spanning tree $T'$ of $\Li G$ and reconstruct a tree array $\langle l_v\rangle$.

\begin{lemma} \label{Ttree} Suppose $T'$ is a spanning tree of $\Li G$ with root $r'$, and that $\pi$ takes $T'$ to the array of lists $\langle l_v \rangle$. Then $\langle l_v \rangle$ is a tree array, which means that
\begin{enumerate}[(a)]
\item The length of $l_v$ is $\indeg(v)$, for all $v\in V$.
\item Every element of $l_v$ is an edge with source $v$, for all vertices $v$ except $t(r')$. The last element of $l_{t(r')}$ is $\Omega$, and every other element of $l_{t(r')}$ is an edge with source $t(r')$.
\item The set $T$ of edges which are the last elements of the lists $\{l_v | v\in V\backslash t(r')\}$ is a spanning tree of $G$.
\end{enumerate}
\end{lemma}

\begin{proof} We first show parts (a) and (b). Each time an edge $e\in E$ is removed from $T'$, an element is appended to the list $l_{t(e)}$. Since $r'$ can only be removed after all the other vertices of $T'$, this algorithm adds $\indeg(v)$ elements to $l_v$ for all $v\in V$, so part (a) holds. Every element of the list $l_v$ is an edge with source $v$, with the exception of $\Omega$, which is the last element of $l_r$, so part (b) holds. 

While the algorithm $\pi$ is running, say a vertex $v\in V$ is \textit{filled} if $l_v$ has $\indeg(v)$ elements. Every vertex is eventually filled, so the order in which vertices are filled is a total order on $V$. 

We claim that this order is a topological sort of the subgraph $T$. Suppose that $f=(v,w)$ is the last element of $l_v$ for some vertex $v$ other than $r$. Vertex $v$ was filled at step 2, right after some leaf $e$ and some edge $(e,f)$ were removed from $T'$ in step 1. However, $w$ cannot be filled until $f$ is removed from $T'$, which happens after $e$ and $(e,f)$ are removed from $T'$. Therefore $v$ is filled before $w$, so the filling ordering on $V$ is a topological sort of $T$, and $T$ is acyclic. 

After $\pi$ terminates, $t(r')$ has no outedge in $T$ and every other vertex of $G$ has one outedge. Then $T$ is a spanning tree of $G$, and part (c) holds. \end{proof} \vspace{.05in}

We now prove the main result.

\begin{proof}[\textbf{Proof of Theorem~\ref{thm1.1}}] Let $\langle l_v \rangle$ be a tree array and let $T'=\sigma(\langle l_v \rangle)$. We show the following claim by induction on $n$: after $n$ edges have been added to $T'$ by the algorithm for $\sigma(\langle l_v \rangle)$, and $n$ edges have been removed from $T'$ by the algorithm for $\pi(T')$, we have
\begin{enumerate}[(a)]
\item The set of edges added to $T'$ by $\sigma$ is the set of edges removed by $\pi$.
\item The elements popped from $l_v$ by $\sigma$ are exactly the elements added to $l_v$ by $\pi$, in the same order.
\end{enumerate} 
In the base case $n=0$, both claims hold trivially. Suppose both results hold for $n=k$. The edge $e$ is a leaf of $T'$ in $\pi$ if and only if it satisfies $N(l_{s(e)},e)=0$ and $\outdeg_{T'}(e)=0$ in $\sigma$. 

Therefore, the $(k+1)$st edge $(f,g)$ added to $T'$ by $\sigma$ is also the $(k+1)$st edge removed from $T'$ by $\pi$, and the element $g$ popped from $l_{t(f)}$ in $\sigma$ is also the element appended to $l_{t(f)}$ by $\pi$. Both claims hold for $n=k+1$. By induction, they hold for all $n\le |E|-1$. 

When $n=|E|-1$, condition (b) implies that $\pi(T')=\langle l_v \rangle$. Then $\pi$ is a left inverse of $\sigma$, and $\sigma$ is injective. 

By similar reasoning, $\pi$ is a right inverse of $\sigma$, and $\sigma$ is surjective. So $\sigma$ is a bijection between tree arrays in $G$ and spanning trees of $\Li G$. The bijection $\sigma$ induces between equal terms in Eq. (\ref{kthm}) proves Theorem~\ref{thm1.1}.
\end{proof}

\section{The de Bruijn bijection} \label{dbsequence} 

A \textit{binary de Bruijn sequence of degree $n$} is a cyclic binary sequence $B$ such that every binary sequence of length $n$ appears as a subsequence of consecutive elements of $B$ exactly once. For example, 0011 is a binary de Bruijn sequence of degree 2, since its cyclic subsequences of length 2 are 00, 01, 11, and 10. 

It is well-known that there are $2^{2^{n-1}}$ binary de Bruijn sequences of degree $n$. Stanley posed the following open problem in \cite{St}. \vspace{.1in}

\noindent \textbf{Exercise 5.73 of \cite{St}.} Let $\mathcal{B}(n)$ be the set of binary de Bruijn sequences of degree $n$, and let $\mathcal{S}_n$ be the set of all binary sequences of length $2^n$. Find an explicit bijection $\mathcal{B}(n)\times \mathcal{B}(n) \rightarrow \mathcal{S}_n$. \vspace{.1in}

Our solution to this problem involves the de Bruijn graphs, which are closely related to de Bruijn sequences.

\begin{definition}[de Bruijn graph] The \textit{de Bruijn graph} $\db_n(m)$ has $m^n$ vertices, which are identified with the strings of length $n$ on $m$ symbols. The edges of the graph are labeled with the strings of length $n+1$ on $m$ symbols. The edge $s_0s_1\ldots s_n$ has source $s_0s_1\ldots s_{n-1}$ and target $s_1s_2\ldots s_n$. \end{definition}

An edge of $\db_n(m)$ can be identified with the vertex of $\db_{n+1}(m)$ that is labeled with the same string of length $n+1$. With this identification, we have
\[\db_n(m)=\Li\db_{n-1}(m)\]
Each vertex $v=s_0s_1\ldots s_{n-1}$ of $DB_n(2)$ has two outedges, $s_1s_2,\ldots s_{n-1}0$ and $s_1s_2\ldots s_{n-1}1$. We call these edges the \textit{zero edge} of $v$ and the \textit{one edge} of $v$, respectively.

It is well-known that binary de Bruijn sequences of degree $n$ are in bijection with Hamiltonian paths in $DB_n(2)$. Let $B=b_0b_1\ldots b_{2^n-1}$ be a binary de Bruijn sequence of degree $n$. Let $v_i=b_ib_{i+1}\ldots b_{i+n-1}$, for $0\le i\le 2^n-1$, where indices are taken mod $2^n$. The path $(v_0,v_1,\ldots v_{2^n-1})$ is the corresponding Hamiltonian path in $DB_n(2)$. 

\begin{theorem} There is an explicit bijection between $\mathcal{B}(n)$ and the set of binary sequences of length $2^{n-1}$, for $n>1$. \end{theorem}

\begin{proof} We describe a bijection between Hamiltonian paths in $\db_n(2)$ and binary sequences of length $2^{n-1}$. By composing this bijection with the map between de Bruijn sequences and Hamiltonian paths, we construct the desired bijection.

We order the vertices in $DB_k(2)$ by the lexicographic order on their associated binary strings, for $1\le k\le n$. Let $(v_1,\ldots v_{2^n})$ be a Hamiltonian path in $DB_n(2)$. This path is an oriented spanning tree of $DB_n(2)$, so we can apply the inverse map $\pi$ defined in Section \ref{bijection} to it.

Let $A_{n-1}$ be the tree array $A_{n-1}=\pi(v_1,\ldots v_{2^n})$. We recursively define a sequence of tree arrays $A_k$, for $1\le k\le n-1$. Suppose we have a tree array $A_{k+1}$ in $DB_{k+1}(2)$. Let $T_{k+1}$ be the spanning tree consisting of the edges which are the last elements of the lists in $A_{k+1}$. We define $A_k$ to be $\pi(T_{k+1})$.

We construct a binary sequence $s_1s_2\ldots s_{2^{n-1}}$ from these tree arrays. We denote vertex $w$'s list in the tree array $A_k$ by $(A_k)_w$. Let $s_{2^{n-1}}$ be 0 if the first element of $(A_{n-1})_{s(v_{2^n})}$ is the zero edge of $s(v_{2^n})$, and 1 otherwise. 

We define $s_{2^k}$ through $s_{2^{k+1}-1}$, for $1\le k\le n-2$, as follows. Let $w_1,w_2,\ldots w_{2^k}$ be the vertices of $\db_k(2)$, in lexicographic order. Let $s_{2^k+i-1}$ be 0 if the first element of $(A_k)_{w_i}$ is the zero edge of $w_i$, and 1 otherwise. Let $s_1$ be 0 if $T_1$ is rooted at vertex 0, and 1 otherwise. 

The string $s_1s_2\ldots s_{2^{n-1}}$ is the binary sequence that corresponds to the Hamiltonian path we began with.

Now we construct the inverse map, from binary sequences to Hamiltonian paths. Given any binary sequence $S$ of length $2^{n-1}$, we use the first $2^{n-1}-1$ characters of the sequence to invert the previous procedure and construct a sequence of spanning trees $T_1,T_2,\ldots T_{n-1}$. The tree $T_k$ will be a spanning tree of $DB_k(2)$.  

We determine $T_k$ recursively. The tree $T_1$ in $DB_1(2)$ is rooted at 0 if $s_1$ is 0, and rooted at 1 otherwise. Assume that the first $2^k-1$ characters of $S$ determine a spanning tree $T_k$ of $DB_k(2)$, where $k\le n-2$. We choose a tree array $A_k$ of $DB_k(2)$ using this tree and the next $2^k$ characters of $S$, as follows. 

Let the vertices of $DB_k(2)$ be $w_1,w_2,\ldots w_{2^k}$, in lexicographic order. The first element of $(A_k)_{w_i}$ is the zero edge of $w_i$ if $s_{2^k+i-1}$ is 0, and the one edge of $w_i$ otherwise. The second element of $(A_k)_{w_i}$ comes from $T_k$. We define $T_{k+1}$ to be $\sigma(A_k)$, using the map defined in Section \ref{bijection}. 

We use $T_{n-1}$ to construct a tree array $A_{n-1}$ such that $\sigma(A_{n-1})$ is a Hamiltonian path in $DB_n(2)$. Let $r$ be the root of $T_{n-1}$, and let $v$ be another arbitrary vertex. The list $l_v$ in the array $A_{n-1}$ must contain two distinct edges, if $\sigma(A_{n-1})$ is a Hamiltonian path. The second edge in $l_v$ must be the unique edge in $T_{n-1}$ with source $v$, so $l_v$ is determined. Our only remaining choice is which of the two edges of $DB_{n-1}(2)$ with source $r$ to include in $l_r$, which we determine by $s_{2^{n-1}}$. 

Clearly, this map from binary sequences to Hamiltonian paths inverts the map from Hamiltonian paths to binary sequences. Therefore, our first map is the bijection we need. \end{proof} \vspace{.05in}

This bijection can easily be generalized to count the $k$-ary de Bruijn sequences, in which the 2-symbol alphabet $\{0,1\}$ is replaced with the $k$-symbol alphabet $\{0,1,\ldots k-1\}$.  

\section{The Kautz and de Bruijn graphs} \label{kautzsection}

In this section, we determine the critical groups of all the Kautz graphs and the de Bruijn graphs. The critical groups of these graphs have been found in some special cases by Levine \cite{Le}. 

The Kautz graphs are similar to the de Bruijn graphs, except that the vertices are indexed by \textit{Kautz strings}. A Kautz string is a string in which no two adjacent characters are the same.

\begin{definition}[Kautz graph] The \textit{Kautz graph} $\kautz_n(m)$ has $(m+1)m^{n-1}$ vertices, identified with the Kautz strings of length $n$ on $m+1$ symbols. The edges of the graph are labeled with the Kautz strings of length $n+1$ on $m+1$ symbols, such that the edge $s_0s_1\ldots s_n$ has source $s_0s_1\ldots s_{n-1}$ and target $s_1s_2\ldots s_n$. \end{definition} 

We also consider the Kautz and de Bruijn graphs as families of iterated line graphs. $\kautz_1(m)$ is the complete directed graph on $m+1$ vertices, without self-loops, and $\db_1(m)$ is the complete directed graph on $m$ vertices, with self-loops. Then for $n>1$, we have
\[\kautz_{n+1}(m)=\Li\kautz_n(m)=\Li^n\kautz_1(m)\]
\[\db_{n+1}(m)=\Li\db_n(m)=\Li^n\db_1(m)\]
We say a directed graph $G=(V,E)$ is \textit{balanced k-regular} if $\indeg(v)=\outdeg(v)=k$ for all $v\in V$. Both $\kautz_n(m)$ and $\db_n(m)$ are balanced $m$-regular, for all $n\in \nn$, which implies that they are Eulerian. Since these graphs are also strongly-connected, their critical groups are defined.

Levine found the critical groups of the de Bruijn graphs $\db_n(2)$ and the Kautz graphs $K_n(p)$, where $p$ is prime \cite{Le}. In this section we characterize the critical groups of all the Kautz and de Bruijn graphs. We prove the following theorems.

\begin{theorem} \label{db} The critical group of $\db_n(m)$ is
\[K\left(\db_n(m)\right)=\left(\zz_{m^n}\right)^{m-2}\oplus\bigoplus_{i=1}^{n-1} \left(\zz_{m^i}\right)^{m^{n-1-i}(m-1)^2}\]
\end{theorem}

\begin{theorem} \label{kautz} The critical group of $\kautz_n(m)$ is
\[K\left(\kautz_n(m)\right)=\left(\zz_{m+1}\right)^{m-1}\oplus \left(\zz_{m^{n-1}}\right)^{m^2-2}\oplus\bigoplus_{i=1}^{n-2} \left(\zz_{m^i}\right)^{m^{n-2-i}(m-1)^2(m+1)}\]
\end{theorem}

In order to prove these theorems, we first prove two lemmas about row-reducing the Laplacians $L(\kautz_n(m))$ and $L(\db_n(m))$. We refer to the row and column of a vertex $v$ in the Laplacian by $R(v)$ and $C(v)$, respectively. We also use $L(v,w)$ to denote the entry in the row of $v$ and the column of $w$. 

We say two strings of length $n$ are \textit{similar} if their last $n-1$ characters are equal. Similarity is an equivalence relation. We partition the vertices of $\kautz_n(m)$ and $\db_n(m)$ into equivalence classes, by grouping vertices labeled with similar strings in the same class. There are $m$ vertices in each class.

\begin{lemma} \label{cycle} Let $G=(V,E)$ be a Kautz graph $\kautz_{n+1}(m)$ or a de Bruijn graph $\db_{n+1}(m)$, where $n\in \nn$. Then $G$ contains a cycle $(v_1,v_2,\ldots v_c)$ of length $c=|V|/m$ which contains one vertex from each class. \end{lemma}

\begin{proof} Let $G'$, the predecessor of $G$, be $\kautz_n(m)$ if $G$ is $\kautz_{n+1}(m)$, and $\db_n(m)$ if $G$ is $\db_{n+1}(m)$. 

First we show that there is a Hamiltonian cycle in $G'$. Such a cycle exists in the complete graphs $K_m$ and $K_{m+1}$, so the case $n=1$ is done. There is an Eulerian tour of $\kautz_{n-1}(m)$ and of $\db_{n-1}(m)$ for $n>1$, since graphs in both families are Eulerian. Because $G'$ is either $\Li \kautz_{n-1}(m)$ or $\Li \db_{n-1}(m)$, one of these Eulerian tours induces a Hamiltonian cycle in $G'$, for $n>1$.

The Hamiltonian cycle in $G'$ can be represented as a string $S=s_1s_2\ldots s_{n+c-1}$, where the $i$th vertex of the cycle is labeled with $s_is_{i+1}\ldots s_{i+n-1}$. 

We use string $S$ to find a cycle in $G$.  Let $v_i=s_is_{i+1}\ldots s_{i+n}$ for $i<c$, and let $v_c=s_cs_{c+1}\ldots s_{n+c-1}s_1$. By the construction of $S$, $(v_1,v_2,\ldots v_c)$ is a cycle which contains one vertex from each class. \end{proof} \vspace{.05cm}

In the next lemma we show that every invariant factor of $L(\kautz_{n+1}(m))$ and $L(\db_n(m))$ is either a multiple of $m$ or relatively prime to $m$. We prove this lemma by row-reducing the Laplacian in an order derived from the cycle in Lemma~\ref{cycle}. 

\begin{lemma} \label{divbym} Let $G=(V,E)$ be a Kautz graph $\kautz_{n+1}(m)$ or a de Bruijn graph $\db_{n+1}(m)$, where $n\in \nn$. The first $c=|V|/m$ invariant factors of $L(G)$ are relatively prime to $m$, and all of the rest are divisible by $m$. \end{lemma} 

\begin{proof} We reduce the Laplacian $L(G)$ over the principal ideal domain $\zz_m$. Let the invariant factors of $L(G)$ over $\zz$ be $x_1,x_2,\ldots x_{|V|}$. Any invertible row or column operation over $\zz$ descends to an invertible operation over $\zz_m$, so the invariant factors of $L(G)$ over $\zz_m$ are the $x_i$ mod $m$.

Let $(v_1,v_2,\ldots v_c)$ be the cycle in $G$ from Lemma~\ref{cycle}, and let $[v_i]$ be the set of vertices in the class of $v_i$. We take indices mod $c$, so $v_{c+1}$ is $v_1$. 

Note that if $u$ and $v$ are vertices in the same class, then $(u,w)$ is an edge if and only if $(v,w)$ is, for all $w$. Therefore, the rows of $u$ and $v$ in the adjacency matrix $A(G)$ are the same.

Because every vertex of $G$ has outdegree $m$, $L(G)\equiv A(G)$ mod $m$. Therefore rows $R(u)$ and $R(v)$ are congruent mod $m$ if $u$ and $v$ are in the same class. We reduce the Laplacian in $c$ stages. In the $i$th stage, we subtract row $R(v_i)$ from $R(v)$ for all $v\in [v_i]\backslash v_i$. After this operation, every entry of $R(v)$ is divisible by $m$. 
{\small \[\mbox{\bordermatrix{% 
    & 01 & 02 & 10 & 12 & 20 & 21 \cr
  01 & -2 & 0 & 1 & 1 & 0 & 0 \cr
  02 & 0 & -2 & 0 & 0 & 1 & 1 \cr
  10 & 1 & 1 & -2 & 0 & 0 & 0 \cr
  12 & 0 & 0 & 0 & -2 & 1 & 1 \cr
  20 & 1 & 1 & 0 & 0 & -2 & 0 \cr
  21 & 0 & 0 & 1 & 1 & 0 & -2 \cr}} \rightarrow \mbox{\bordermatrix{% 
    & 01 & 02 & 10 & 12 & 20 & 21 \cr
  01 & -2 & 0 & 1 & 1 & 0 & 0 \cr
  02 & 0 & -2 & 0 & 2 & 0 & 0 \cr
  10 & 0 & 0 & -2 & 0 & 2 & 0 \cr
  12 & 0 & 0 & 0 & -2 & 1 & 1 \cr
  20 & 1 & 1 & 0 & 0 & -2 & 0 \cr
  21 & 2 & 0 & 0 & 0 & 0 & -2 \cr}} \] }
\begin{center} \textit{Figure 5.1. Reducing $L(\kautz_2(2))$ using the cycle $(01,12,20)$. The original Laplacian is on the left. We obtain the reduced Laplacian on the right by subtracting $R(01)$ from $R(21)$, $R(12)$ from $R(02)$, and $R(20)$ from $R(10)$. Every entry of rows $R(02)$, $R(10)$, and $R(21)$ is divisible by 2.}\end{center} 

The entry $L(v_i,v_{i+1})$ is 1 before and after these row operations, for $1\le i\le c$. We claim that in the reduced Laplacian, every entry of $C(v_{i+1})$ is divisible by $m$ except $L(v_i,v_{i+1})$. There are $m$ edges with target $v_{i+1}$ in $G$. The sources of these edges are the $m$ vertices in $[v_i]$. After the row operations, every entry of $R(v)$ is divisible by $m$ for $v\in [v_i]\backslash v_i$, so $L(v_i,v_{i+1})$ is the only entry in $C(v_{i+1})$ which is non-zero mod $m$, for $1\le i\le c$. 

By permuting rows and columns, we move $L(v_i,v_{i+1})$ to the $i$th diagonal entry $L_{ii}$ of the Laplacian. The reduced Laplacian is now in the form
\[\begin{pmatrix}I_c & A \\ 0 & 0\end{pmatrix}\ \text{mod } m\]
where $I_c$ is the $c\times c$ identity matrix. Using column operations, we can make all the entries in A divisible by $m$, without changing the rest of the matrix mod $m$. After this column operations, the Laplacian is in Smith normal form.

The first $c$ invariant factors of $L(G)$ over $\zz_m$ are 1, so the first $c$ invariant factors of $L(G)$ over $\zz$ are relatively prime to $m$. The last $|V|-c$ invariant factors of $L(G)$ over $\zz_m$ are 0, so the last $|V|-c$ invariant factors of $L(G)$ over $\zz$ are divisible by $m$. The lemma holds.
\end{proof}

We use Lemmas~\ref{cycle} and \ref{divbym} to characterize the critical group of the Kautz and de Bruijn graphs. The first step is finding the orders of these groups. We apply Theorem~\ref{thm1.1} to $\db_n(m)$, and we let all the variables $x_e$ equal 1, to find that
\[\kappa\left(\db_{n+1}(m)\right)=\kappa\left(\db_n(m) \right)\left( m^{(m-1)m^n}\right)\]
The number of spanning trees of the complete graph $\db_1(m)$ is $m^{m-1}$ \cite{Ho}. By a simple induction, we have
\[\kappa\left(\db_n(m)\right)=m^{m^n-1}\]
In an Eulerian graph, the sandpile groups $K(G,v)$ are isomorphic for all vertices $v$, so $|K(G)|=\kappa(G)/|V|$. Therefore, we have
\begin{equation}\label{dborder}  |K\left(\db_n(m)\right)|=m^{m^n-n-1}\end{equation}
Similarly, we have
\[\kappa\left(\kautz_n(m)\right)=(m+1)^mm^{\left(m^n-1\right)(m+1)}\]
\begin{equation}\label{order}|K\left(\kautz_n(m)\right)|=(m+1)^{m-1}m^{\left(m^n+m^{n-1}-m-n\right)}\end{equation}
We are ready to prove theorems~\ref{db} and ~\ref{kautz}.

\begin{proof}[\textbf{Proof of Theorem~\ref{db}}] We proceed by induction on $n$. The critical group of the complete graph on $m$ vertices is $(\zz_m)^{m-2}$, so the base case holds.

Assume that Theorem~\ref{db} holds for $n-1$, where $n>1$. We prove it for $n$. As shown by Levine \cite{Le}, if $G$ is a balanced $k$-regular graph, then
\begin{equation} \label{homo} k K(\Li G) \cong K(G)\end{equation}
We will use this fact to determine $Syl_p(K(\db_n(m)))$, the Sylow-$p$ subgroup of $K(\db_n(m))$, for any prime $p$. We break into two cases: either $p$ does not divide $m$, or $p$ divides $m$.

If $p$ does not divide $m$, then by Eq. (\ref{homo}), we have 
\[\text{Syl}_p\left(K\left(\db_n(m)\right)\right)\cong \text{Syl}_p\left(K\left(\db_{n-1}(m)\right)\right)\] 
By the inductive hypothesis, $\text{Sylow}_p(K(\db_n(m)))$ is the trivial group. 

Now let $p$ be a prime that divides $m$, and suppose $p^k$ is the largest power of $p$ that divides $m$. Let the Sylow-$p$ subgroup of $\db_n(m)$ be 
\[\text{Syl}_p\left(K\left(\db_n(m)\right)\right) = \zz_p^{a_1}\oplus \zz_{p^2}^{a_2} \oplus \ldots \oplus \zz_{p^l}^{a_l}\]
By Lemma~\ref{divbym}, $K(\db_n(m))$ can be written as a direct sum of cyclic groups, such that the order of each group is either non-zero mod $p$ or divisible by $p^k$. Thus, $a_i=0$ for $i<k$. Further, we can derive the order of $\text{Syl}_p(K(\db_n(m)))$ from Eq. (\ref{dborder}). We find that
\begin{equation}\label{dbporder} \sum_{i=n}^l ia_i = k\left(m^n-n-1\right) \end{equation}
because the expression on the right-hand side equals the number of factors of $p$ in $m^{m^n-n-1}$. By Eq. (\ref{homo}) and the inductive hypothesis, we know that
\[p^k\text{Syl}_p\left(K\left(\db_n(m)\right)\right)=\zz_p^{a_{k+1}}\oplus \zz_{p^2}^{a_{k+2}} \oplus \ldots \oplus \zz_{p^{l-k}}^{a_{l}} \cong \text{Syl}_p\left(K\left(\db_{n-1}(m)\right)\right)\]
\begin{equation}\label{dbindstep} \zz_p^{a_{k+1}}\oplus \zz_{p^2}^{a_{k+2}} \oplus \ldots \oplus \zz_{p^{l-k}}^{a_{l}} \cong \left(\zz_{p^{nk}}\right)^{m-2}\oplus \bigoplus_{i=1}^{n-2}\left(\zz_{p^{ik}}\right)^{m^{n-2-i}(m-1)^2} \end{equation}
Eq. (\ref{dbindstep}) implies that $a_{nk}=m-2$, that $a_{(i+1)k}=m^{n-2-i}(m-1)^2$ for $1\le i \le n-2$, and that $a_i=0$ for $i>nk$ or $k\nmid i$. The only $a_i$ which we have not yet determined is $a_k$. We solve Eq. (\ref{dbporder}) for $a_k$, by moving all the other $a_{ik}$ to the right-hand side and dividing by $k$.
\[a_k = \left(m^n-n-1\right)-n(m-2)-\sum_{i=2}^{n-1} i\left(m^{n-1-i}(m-1)^2\right)\]
By evaluating the geometric series, we find that $a_k=m^{n-2}(m-1)^2$. With these values, we may write
\[\text{Syl}_p\left(K\left(\db_n(m)\right)\right)=\left(\zz_{p^{nk}}\right)^{m-2}\oplus \bigoplus_{i=1}^{n-1}\left(\zz_{p^{in}}\right)^{m^{n-1-i}(m-1)^2}\]
The Sylow-$p$ subgroups of $K(\db_n(m))$ are trivial for $p\nmid m$. Taking the direct sum of the Sylow-$p$ subgroups over primes $p$ which divide $m$, we find
\[ K\left(\db_n(m)\right)\cong \bigoplus_{p\mid m}\text{Sylow}_p\left(K\left(\db_n(m)\right)\right)\cong\left(\zz_{m^n}\right)^{m-2}\oplus\bigoplus_{i=1}^{n-1} \left(\zz_{m^i}\right)^{m^{n-1-i}(m-1)^2} \]
With this equation we complete the inductive step, as desired.
\end{proof} \vspace{.05in}

\begin{proof}[\textbf{Proof of Theorem~\ref{kautz}}] This proof is similar to the proof of Theorem~\ref{db}. Again, we induct on $n$. Because the critical group of the complete graph on $m+1$ vertices is $(\zz_{m+1})^{m-1}$, the base case holds.

Assume that Theorem~\ref{kautz} holds for $n-1$, where $n>1$. Using Eq. (\ref{homo}), we calculate the direct sum of the Sylow-$p$ subgroups of $K(\kautz_n(m))$ over primes $p$ which do not divide $m$, as follows
\begin{equation}\label{nodivide} \bigoplus_{p\nmid m} \text{Syl}_p\left( K\left(\kautz_n(m)\right) \right)= \bigoplus_{p\nmid m} \text{Syl}_p \left(K\left(\kautz_1(m)\right)\right) = \left(\zz_{m+1}\right)^{m-1} \end{equation} 
Now let $p$ be a prime that divides $m$. Suppose that $p^k$ is the largest power of $p$ that divides $m$, and that the Sylow-$p$ subgroup of $K(\kautz_n(m))$ is
\[\text{Syl}_p\left(K\left(\kautz_n(m)\right)\right) = \zz_p^{a_1}\oplus \zz_{p^2}^{a_2} \oplus \ldots \oplus \zz_{p^l}^{a_l}\]
Lemma~\ref{divbym} implies that $a_i=0$ for $i<k$. Furthermore, we know the order of the Sylow-$p$ subgroup of $K(\kautz_n(m))$ from Eq. (\ref{order}), which implies that 
\begin{equation}\label{porder} \sum_{i=k}^l ia_i = k\left(m^n+m^{n-1}-m-n\right) \end{equation}
because the expression on the right-hand side equals the number of factors of $p$ in $m^{(m^n+m^{n-1}-m-n)}$. By Eq. (\ref{homo}), we have
\[p^k\text{Syl}_p\left(K\left(\kautz_n(m)\right)\right)=\zz_p^{a_{k+1}}\oplus \zz_{p^2}^{a_{k+2}} \oplus \ldots \oplus \zz_{p^{l-k}}^{a_{l}} \cong \text{Syl}_p\left(K\left(\kautz_{n-1}(m)\right)\right)\]
By the inductive hypothesis, we find that $a_{(i+1)k}={m^{n-3-i}(m-1)^2(m+1)}$ for $1\le i \le n-3$, that $a_{(n-1)k}=m^2-2$, and that $a_i=0$ for $i>(n-1)k$ or $k\nmid i$. 

We solve Eq. (\ref{porder}) for $a_k$. We find that $a_k=m^2-2$ if $n=2$ and that $a_k=m^n-m^{n-1}-m^{n-2}+1$ if $n>2$. Then we have
\begin{equation} \label{divide} \text{Syl}_p\left(K\left(\kautz_n(m)\right)\right)=\left(\zz_{p^{(n-1)k}}\right)^{m^2-2}\oplus \bigoplus_{i=1}^{n-2}\left(\zz_{p^{ik}}\right)^{m^{n-2-i}(m-1)^2(m+1)} \end{equation}
By taking the direct sum of Eq. (\ref{nodivide}) and Eq. (\ref{divide}) over all primes which divide $m$, we complete the inductive step and prove Theorem~\ref{kautz}, as desired.
\end{proof}
\end{document}